\documentclass[10pt, a4paper]{article}
\usepackage{amssymb}
\usepackage{amsmath}
\usepackage{amsfonts}
\usepackage{amscd}
\usepackage{geometry}

\newcommand{\C}{\mathbb{C}}
\newcommand{\R}{\mathbb{R}}
\newcommand{\Q}{\mathbb{Q}}
\newcommand{\N}{\mathbb{N}}
\newcommand{\Z}{\mathbb{Z}}
\newcommand{\A}{\mathbb{A}}

\newcommand{\F}{\mathbb{F}}
\newcommand{\n}{\mathfrak{n}}
\newcommand{\p}{\mathfrak{p}}
\newcommand{\q}{\mathfrak{q}}

\DeclareMathOperator{\proj}{proj}
\DeclareMathOperator{\Ind}{Ind}
\DeclareMathOperator{\Frob}{Frob}
\DeclareMathOperator{\GL}{GL}
\DeclareMathOperator{\Tr}{Tr}
\DeclareMathOperator{\PGL}{PGL}
\DeclareMathOperator{\PSL}{PSL}
\DeclareMathOperator{\Gal}{Gal}

\DeclareMathOperator{\SL}{SL}

\usepackage{amsthm}
\theoremstyle{plain}
\newtheorem{theo}{Theorem}[section]
\newtheorem{prop}[theo]{Proposition}
\newtheorem{lemm}[theo]{Lemma}
\newtheorem{coro}[theo]{Corollary}
\newtheorem*{tio}{Theorem}
\theoremstyle{remark}
\newtheorem{rema}[theo]{\sc Remark}
\theoremstyle{definition}
\newtheorem{defi}[theo]{Definition}

\title{Constructing  Hilbert modular forms whithout exceptional primes}
\author{\small LUIS V. DIEULEFAIT \footnote{Departament d'Algebra i Geometria,
 Facultat de Matem$\grave{\mbox{a}}$tiques, Universitat de Barcelona,
  Gran Via de les Corts Catalanes, 585, 08007 Barcelona, Spain, \texttt{ldieulefait@ub.edu}},
   ADRI$\acute{\mbox{A}}$N ZENTENO-GUTI$\acute{\mbox{E}}$RREZ  \footnote{Instituto de Matem$\acute{\mbox{a}}$ticas,
    (Unidad Cuernavaca) UNAM. Av. Universidad s/n. Col. Lomas de Chamilpa CP 62210,
     Cuernavaca, Mexico. \texttt{matematicazg@ciencias.unam.mx}}}

\date{\today}
\begin{document}

\maketitle

\begin{abstract}
In this paper we construct families of Hilbert modular newforms without exceptional primes. This is achieved by generalizing the notion of good-dihedral primes, introduced by Khare and Wintenberger in their proof of Serre's modularity conjecture, to totally real fields.

2010 \textit{Mathematics Subject Classification}. 11F80, 11F41.
\end{abstract}

\section*{Introduction}
\label{intro}
In order to prove new cases of the Inverse Galois Problem, Dieulefait and Wiese \cite{diu} have constructed families of newforms without exceptional primes by using the notion of tamely dihedral representation, which is a slight variation of the  good-dihedral representations introduced by Khare and Wintenberger in \cite{kw1}. More precisely they prove the following result.

\begin{tio}
There exist eigenforms $(f_n)_{n \in \N}$ with  coefficient fields $(E_{f_n})_{n\in \N}$ of weight 2,
 trivial nebentypus, without nontrivial inner twists and without complex multiplication such that
\begin{enumerate}
\item for all $n$ and all maximal ideals $\Lambda_n$ of $\mathcal{O}_{E_{f_n}}$ the Galois representation
 $\overline{\rho}^{\proj}_{f_n, \Lambda_n}$ associated to $f_n$ is nonexceptional in the sense
  of Definition \ref{exep} and
\item for fixed prime $\ell$, the size of the image
of $\overline{\rho}^{\proj}_{f_n, \Lambda_n}$ for $\Lambda_n | \ell$ is unbounded for running $n$.
\end{enumerate}
\end{tio}
In this paper we extend this construction to Hilbert modular newforms of arithmetic weight $k$ over a totally real field $F$, defining the notion of tamely dihedral representation in $F$.

Our construction closely follows the construction of Dieulefait and Wiese which consists of adding tamely dihedral primes to the level, via level raising, which correspond to supercuspidal representations. By using a Lemma of Dimitrov we are going to be able to construct Hilbert modular newforms of arbitrary
weight, not only of weight $2$ as in \cite{diu}. On the other hand, if $F$ is a Galois field of odd degree, we will add an extra ingredient  to our construction in order to avoid the possibility that the Hilbert modular newforms considered come from a base change. This phenomenon does not occur in the classical case.

We itemize the contents of the paper. In the first two sections we recall the notion of inner twist and complex multiplication for $2$-dimensional Galois representations. In section 3 we review essential facts about the correspondence between Galois
representations and Hilbert modular forms. In section 4 we define the notion of tamely dihedral representations and show some basic lemmas that will later be used in order to exclude nontrivial inner
twists. In section 5 we collect the tools for constructing the sought for Hilbert modular newforms via level raising. Finally, in the last two sections we construct families of Hilbert modular newforms without exceptional primes.

\paragraph*{Acknowledgements} 
This paper is part of the second author's PhD thesis.
Part of this work has been written during a stay at the Hausdorff Research Institute for Mathematics and a stay, of the second author, at the Mathematical Institute of the University of Barcelona. The authors would like to thank these institutions for their
support and the optimal working conditions. The research of L.V.D. was supported by an ICREA Academia Research Prize and by MICINN  grant MTM2012-33830. The
research of A.Z. was supported by the CONACYT grant no. 432521/286915.

\paragraph*{Notation}
We shall use the following notations. Let $K$ be a number field or a $\ell$-adic field. We denote by $\mathcal{O}_K$ its ring of integers. In particular if $K$ is a number field, $\mathfrak{d}_K$ denote the different of $K$, $h_K$ denote the class number of $K$ and by a prime of $K$ we mean a nonzero prime ideal of $\mathcal{O}_K$. We denote by $G_K$ the absolute Galois group of $K$. For a maximal ideal $\q$ of $\mathcal{O}_K$, we let $D_{\q}$ and $I_{\q}$ be the corresponding decomposition and inertia group at $\q$, respectively. By $S_k(\n,\psi)$ we means the space of Hilbert modular cuspforms of weight $k$, level $\n$ and central character $\psi$. If $\psi$ is trivial, we write $S_k(\n)$ for short.  Finally for a number field $K$ we denote by $K(\mathfrak{\p}_1, \ldots , \p_m)$ the maximal abelian polyquadratic extension which is ramified only at the primes $\p_1,
\ldots, \p_m$. By the Hermite-Minkowski theorem $K(\mathfrak{\p}_1, \ldots , \p_m)$ is a number field.

%%%%%%%%%%%%%%%%%%%%%%%%%%%%%%%%%%%%%%%%%%%%%%%%%%%%%%%%%%%%%%%%%%%%%%%%%

\section{Inner twists and complex multiplication}
\label{sec:1}

In this and the next section we review some facts on inner twists and complex multiplication for $2$-dimensional Galois representations. Our main reference is \cite{adw}.

Let $K$ be an $\ell$-adic field with the $\ell$-adic topology or a finite field with the discrete topology and $L/K$ a finite Galois extension with Galois group $\Gamma := \Gal(L/K)$ endowed with the Krull topology.
Let $F$ be a totally real field and $\mathcal{E} = \{ \epsilon :G_F \rightarrow L^{\times}\}$ be the set of continuous characters from $G_F$ to $L^{\times}$.
Note that our assumptions imply that the image of  $\epsilon$ lies in a finite extension of $K$ and that $\Gamma$ acts on $\mathcal{E}$ on the left by  composition: $^{\gamma}\epsilon := \gamma \circ \epsilon$. Then, we can form the semi-direct product $ \mathcal{G}: = \mathcal{E} \rtimes \Gamma $ induced by this action. Concretely, the product and inverse in $ \mathcal{G} $ are defined as:
\[
(\gamma_1, \epsilon_1) \cdot (\gamma_2, \epsilon_2) := (\gamma_1 \gamma_2 , (^{\gamma_1} \epsilon_2)\epsilon_1)  \quad \mbox{and}
 \quad (\gamma, \epsilon)^{-1}:= (\gamma^{-1}, ^{\gamma^{-1}}(\epsilon^{-1})).
\]
Consequently, we have the exact sequence:
\[
1 \rightarrow \mathcal{E} \stackrel{i}{\rightarrow} \mathcal{G} \stackrel{\pi}{\rightarrow} \Gamma \rightarrow 1
\]
where $i$ (resp. $\pi$) is defined as $\epsilon \mapsto (1, \epsilon)$ (resp. $(\gamma, \epsilon) \mapsto \gamma$).

In this paper we will consider only $2$-dimensional Galois representations $\rho :G_F \rightarrow \GL_2(L)$, however the next results of Arias-de-Reyna, Dieulefait and Wiese are valid for arbitrary $n$-dimensional Galois representations.

Let $\rho, \rho' :G_F \rightarrow \GL_2(L)$ be Galois representations. We say that $\rho$ and $\rho '$ are equivalent (denoted $\rho \sim \rho'$) if they are conjugate by an element of $\GL_2(L)$.
We denote the set of equivalence classes by $\mathcal{GL}_2(G_F,L)$ and note that $\mathcal{G}$ acts on $\mathcal{GL}_2(G_F,L)$ from the left as follows:
\[
(\gamma, \epsilon)\cdot [\rho] := [(^{\gamma}\rho) \otimes_L \epsilon^{-1}],
\]
where $\gamma \in \Gamma$, $\epsilon\in \mathcal{E}$ and $[\rho] \in \mathcal{GL}_2(G_F,L)$.

Let $[\rho]\in \mathcal{GL}_2(G_F,L)$.
Define $ \mathcal{G}_{[\rho]}$ to be the stabilizer group of $[\rho]$ in $\mathcal{G}$ under the $\mathcal{G}$-action on $\mathcal{GL}_2(G_F,L)$.
This group is called the \emph{inner twists group of} $[\rho]$. Explicitly,  $(\gamma, \epsilon) \in \mathcal{G}$ is an \emph{inner twist of} $[\rho]$ if and only if $[\rho]=[(^{\gamma} \rho) \otimes_L \epsilon^{-1}]$, which is the case if and only if
\[
[^{\gamma} \rho] = [\rho \otimes_L \epsilon].
\]
In particular, if $[\rho], [\rho '] \in \mathcal{GL}_2(G_F,L)$ are absolutely irreducible and such that $\Tr(\rho(g)) = \Tr(\rho'(g))$ for all $g\in G_F$, then $[\rho]=[\rho']$. Therefore, we have for all $[\rho] \in \mathcal{GL}_2(G_F,L)$ absolutely irreducible that
\[
 \mathcal{G}_{[\rho]} = \{ (\gamma, \epsilon) \in \mathcal{G} : \gamma(\Tr(\rho(g))) =  \Tr(\rho(g))\epsilon(g), \; \forall g \in G_F \}.
\]
We define the groups:
\[
\Gamma_{[\rho]}:= \pi(\mathcal{G}_{[\rho]}) \subseteq \Gamma, \quad \mbox{and} \quad \mathcal{E}_{[\rho]}:= i^{-1}(\mathcal{G_{[\rho]}}) = i^{-1}(\ker(\pi | _{\mathcal{G}_{[\rho]}})).
\]

We define the group $ \Delta_{[\rho]} := \{ \gamma \in \Gamma_{[\rho]} : (\gamma, 1) \in \mathcal{G}_{[\rho]} \}$ and the field $E_{[\rho]} := L^{\Delta_{[\rho]}}$.
Let $[\rho] \in \mathcal{GL}_2(G_F, L)$ be (residually) absolutely irreducible.
Let $\chi : G_F \rightarrow K^{\times}$ be any character and $\psi: G_F \rightarrow L^{\times}$ be a character of finite order.
Assume that $\det \rho = \psi \chi$ and $E_{\rho}$ contains the square roots of the values of $\psi$.
Then the equivalence class $[\rho]$ contains a representation that can be defined over the field $E_{[\rho]}$ and $E_{[\rho]}$ is the smallest such subfield of $L$.
Moreover, $E_{[\rho]}$ is generated over $K$ by the traces $\Tr(\rho(g))$ for $g\in G_F$.
Consequently, $E_{[\rho]}$ is called the \emph{field of definition of }$[\rho]$.

Now, let $\rho_1, \rho_2 : G_F \rightarrow \PGL_2(L)$ be projective representations.
We call $\rho_1$ and $\rho_2$ equivalent (also denoted $\rho_1 \sim \rho_2$) if they are conjugate by the class (modulo scalars) of a matrix in $\GL_2(L)$. The equivalence classes of $\rho_i$ is also denoted $[\rho_i]$ and the set of such equivalence class is denoted by $\mathcal{PGL}_2(G_F,L)$.
In particular for $\rho: G_F \rightarrow \GL_2(L)$ we denote by $\rho^{\proj}$ the composition of $\rho$ with the natural projection $\GL_2(L) \rightarrow \PGL_2(L)$.

For $[\rho] \in \mathcal{GL}_2(G_F,L)$ and $\epsilon \in \mathcal{E}$ one has  $\rho^{\proj} \sim (\rho \otimes \epsilon)^{\proj}$. Conversely, if $[\rho_1], [\rho_2] \in \mathcal{GL}_2(G_F,L)$ are such that $\rho^{\proj}_1 \sim \rho^{\proj}_2$, then there is $\epsilon \in \mathcal{E}$ such that $[\rho_1]\sim [\rho_2 \otimes \epsilon]$. Thus we have that
\[
\Gamma_{[\rho]} = \{ \gamma \in \Gamma : \; ^{\gamma}\rho^{\proj} \sim \rho^{\proj} \}
\]
Define the field $K_{[\rho]} = L^{\Gamma_{[\rho]}}$.
If $\rho^{\proj}$ factors as $G_F \rightarrow \PGL_2(\tilde{K}) \rightarrow \PGL_2(L)$ for some field $K \subseteq \tilde{K} \subseteq L$, then $K_{[\rho]} \subseteq \tilde{K}$. Moreover, if $[\rho]$ is such that its restriction to the subgroup
\[
I_{[\rho]}:= \bigcap_{ \{\epsilon \in \mathcal{E} : \exists (\gamma,\epsilon) \in
\mathcal{G}_{[\rho]} \} }\ker (\epsilon)
\]
is (residually) absolutely irreducible (in particular, this implies that $[\rho]$ has no complex multiplication), then the equivalence class of $\rho^{\proj}$ has a member that factors through
$\PGL_2(K_{[\rho]})$ and  $K_{[\rho]}$ is the smallest subfield of $L$ with this property.
Consequently $K_{[\rho]}$ is called the \emph{ projective field of definition of }$[\rho]$.

\begin{rema}
Note that if $E_{[\rho]}$ contains the square roots of $\psi$, then $\Delta_{[\rho]}$ is an open normal
 subgroup of $\Gamma_{[\rho]}$ and, hence, $E_{[\rho]}/K_{[\rho]}$ is a finite extension with Galois
  group $\Gamma_{[\rho]} / \Delta_{[\rho]}$. In particular, if $[\rho]$ does not have any nontrivial
   inner twist and no complex multiplication, $L= E_{[\rho]} = K_{[\rho]}$.
\end{rema}

Finally, we will give a couple of lemmas similar to Proposition 3.3 of \cite{adw}.

\begin{lemm}\label{mio}
Let $K$ be a finite field of characteristic $\ell$ and $[\rho] \in \mathcal{GL}_2(G_F,L)$. Let $\mathfrak{L}$ be a prime of $F$ above $\ell$ and $h$, $t$ two integers. Suppose that
\[
(\rho \otimes \chi_{\ell}^t) | _{I_{\mathfrak{L}}} \simeq \left(
  \begin{array}{cc}
    \varphi_{2h}^b & * \\
    0      & \varphi_{2h}^{b\ell^h} \\
  \end{array}
\right)
\]
where $\varphi_{2h}$ is a fundamental character of niveau $2h$, $\chi_{\ell}$ is the mod-$\ell$ cyclotomic character and $b = a_0+a_1 \ell + \ldots + a_{2h-1} \ell^{2h-1}$ is such that $0 \leq a_i < \frac{\ell-1}{4}$ and $a_0+a_h = \ldots = a_{h-1}+a_{2h-1}$. Then the character $\epsilon$ is unramified at $\mathfrak{L}$ for all $(\gamma,\epsilon) \in \mathcal{G}_{[\rho]}$.
\end{lemm}
\begin{proof}
Note that the restriction to $I_{\mathfrak{L}}$ of the determinant of $\rho$ is $\chi_{\ell}^{a-t}$, where $a=a_i + a_{h+i}$ ($i= 0, \cdots, h-1$).
We know that any exponent $x$ of $\varphi_{2h}$ is of the form 
\[
\sum_{j=1}^{2h}a^{(j)}\ell^{j-1},
\] 
where $\{ a^{(j)} : j = 1,2, \ldots, 2h-1 \}$ is a cyclic permutation of the elements of $S=\{a_0, \ldots, a_{2h-1}\}$. Then we have the following estimate for  $x$:
\begin{equation}\label{cot}
0 \leq x = \sum_{j=1}^{2h}a^{(j)}\ell^{j-1} < \sum_{j=1}^{2h}\frac{\ell -1}{4}\ell^{j-1} = \frac{\ell^{2h}-1}{4}.
\end{equation}

Let $(\gamma, \epsilon) \in \mathcal{G}_{[\rho]}$.  As $K$ is a finite field, $\gamma$ acts by raising to the $\ell^c$-th power for some $c$. In particular $^{\gamma} \varphi_{2h} = \varphi_{2h}^{\ell^c}$. This shows that $^{\gamma}(\rho \otimes \chi_{\ell}^t)|_{I_{\mathfrak{L}}}$ has the same shape as $(\rho \otimes \chi_{\ell}^t)|_{I_{\mathfrak{L}}}$ except that the elements of $S$ are permuted.
Taking the determinant on both sides of $^{\gamma}\rho \cong \rho \otimes \epsilon$ yields that $\epsilon |_{I_{\mathfrak{L}}}$ has order dividing $2$, as $\gamma$ acts trivially on the cyclotomic character.
Moreover, looking at any diagonal entry we get
\begin{equation}\label{peq}
\varphi_{2h}^x = \varphi_{2h}^y \cdot \epsilon |_{I_{\mathfrak{L}}}
\end{equation}
for some exponents $x$ and $y$. 

If we assume that the order of $\epsilon|_{I_{\mathfrak{L}}}$ is $2$, then we have 
\[
\epsilon|_{I_{\mathfrak{L}}} = \varphi_{2h}^{\frac{\ell^{2h}-1}{2}}.
\]
But, as the order of $\varphi_{2h}$ is $\ell^{2h}-1$,  equation (\ref{peq}) implies that $\frac{\ell^{2h}-1}{2} +y -x$ is divisible by  $\ell^{2h}-1$. Thus we get a contradiction because $y$ and $x$ are less than $\frac{\ell^{2h}-1}{4}$ by (\ref{cot}). Therefore $\epsilon|_{I_{\mathfrak{L}}}$ is trivial, then unramified at $\mathfrak{L}$. 
\end{proof}

\begin{lemm}\label{mio2}
Let $K$ be a finite field of characteristic $\ell$ and $[\rho] \in \mathcal{GL}_2(G_F,L)$. Let $\mathfrak{L}$ be a prime of $F$ above $\ell$ and $h$, $t$ two integers. Suppose that
\[
(\rho \otimes \chi_{\ell}^t)| _{I_{\mathfrak{L}}} \simeq
\left(
  \begin{array}{cc}
    \varphi_{h}^a & * \\
    0      & \varphi_{h}^{b} \\
  \end{array}
\right)
\]
where $\varphi_{h}$ is a fundamental character of niveau $h$, $\chi_{\ell}$ is the mod-$\ell$ cyclotomic character and $a$, $b$ are of the form $a = a_0+a_1 \ell + \ldots + a_{h-1} \ell^{h-1}$ and $b = a_h+a_{h+1} \ell + \ldots + a_{2h-1} \ell^{h-1}$ with $0 \leq a_i < \frac{\ell-1}{4}$ and  $a_0+a_h = \ldots = a_{h-1}+a_{2h-1}$. Then the character $\epsilon$ is unramified at $\mathfrak{L}$ for all $(\gamma,\epsilon) \in \mathcal{G}_{[\rho]}$.
\end{lemm}
\begin{proof}
The proof is analogous to the proof of Lemma \ref{mio}. 
\end{proof}

We will say that the representations in Lemma \ref{mio} and Lemma \ref{mio2} have \emph{tame inertia weights at most }$k$ if $a_i \leq k$ for all $i$.

%%%%%%%%%%%%%%%%%%%%%%%%%%%%%%%%%%%%%%%%%%%%%%%%%%%%%%%%%%%%%%%%%%%%%%%%%%

\section{Compatible systems of Galois representations}
\label{sec:2}

Let $L/K$ be a finite Galois extension of number fields, $F$ be a totally real field and $S$  be a finite set of primes of $F$.
A \emph{compatible system} $\rho_{\bullet} = (\rho_{\Lambda})_{\Lambda}$\emph{ of }$2$-\emph{dimensional representation of} $G_F$ consists of the following data:
\begin{itemize}
\item For each prime $\p$ of $F$ not in $S$, a monic polynomial $P_{\p} \in \mathcal{O}_L[X]$.
\item For each prime $\Lambda$ of $L$ (together with fixed embeddings $L \hookrightarrow L_{\Lambda} \hookrightarrow \overline{L}_{\Lambda}$) a continuous Galois representation
\[
\rho_{\Lambda}:G_F \rightarrow \GL_2(\overline{L}_{\Lambda})
\]
such that $\rho_{\Lambda}$ is unramified outside $S \cup S_{\ell}$ (where $\ell$ is the residual characteristic of $\Lambda$ and $S_{\ell}$ is the set of primes of $F$ lying above $\ell$) and such that for all $\p \notin S \cup S_{\ell}$ the characteristic polynomial of $\rho_{\Lambda}(\Frob_{\p})$ is equal to $P_{\p}$ (inside $\overline{L}_{\Lambda}[X]$).
\end{itemize}

Let $a \in \Z$ and $\psi: G_F \rightarrow L^{\times}$ be a continuous finite order character. We say that the compatible system $\rho_{\bullet}=(\rho_{\Lambda})_{\Lambda}$ has \emph{determinant} $\psi \chi^a_{\ell}$ if for all primes $\Lambda$ of $L$ above a rational prime $\ell$, for all primes $\ell$, the determinant of $\rho_{\Lambda}$ is $\psi \chi^a_{\ell}$ with $\chi_{\ell}$ the $\ell$-adic cyclotomic character.

Henceforth, we will assume that $\rho_{\bullet}=(\rho_{\Lambda})_{\Lambda}$ is \emph{almost everywhere absolutely irreducible}, i.e. all its members $\rho_{\Lambda}$ are absolutely irreducible except for finitely many primes $\Lambda$ of $L$. Additionally, we assume that $\rho_{\bullet}$ has determinant $\psi \chi^a_{\ell}$.

Note that if we have a number field $K$ with the discrete topology and $L/K$ a finite Galois extension we can define $\mathcal{E}$, $\Gamma$ and $\mathcal{G}$ in the same way as in Section \ref{sec:1} for $\ell$-adic or finite fields.
For a prime $\p$ of $F$ not in $S$, denote by $a_{\p}$ the coefficient in front of $X$ of $P_{\p}$. We define
\[
\mathcal{G}_{\rho_{\bullet}} := \{ (\gamma, \epsilon) \in \mathcal{G} :
 \gamma(a_{\p}) = a_{\p}\cdot \epsilon(\Frob_{\p}),\; \forall \p \notin S \},
\]
\[
\Gamma_{\rho_{\bullet}} := \pi (\mathcal{G}_{\rho_{\bullet}}) \subseteq
 \Gamma, \quad  \mathcal{E}_{\rho_{\bullet}}:=i^{-1}(\mathcal{G}_{\rho_{\bullet}}) =
  i^{-1}(\ker(\pi|_{\mathcal{G}_{\rho_{\bullet}}}))
\]
and
\[
\Delta_{\rho_{\bullet}}:= \{ \gamma \in \Gamma_{\rho_{\bullet}}: (\gamma, 1) \in \mathcal{G}_{\rho_{\bullet}} \}.
\]
We say that the compatible system $\rho_{\bullet}$ has  \emph{no complex multiplication} if $\mathcal{E}_{\rho_{\bullet}} = \{ 1 \}$.
The field $E_{\rho_{\bullet}}:= L^{\Delta_{\rho_{\bullet}}}$ (resp. $K_{\rho_{\bullet}}:=  L^{\Gamma_{\rho_{\bullet}}}$) is called the \emph{field of definition} (resp. \emph{projective field of definition}) of $\rho_{\bullet}$.

If $E_{\rho_{\bullet}}$ contains the square roots of the values of $\psi$ one has that $\Delta_{\rho_{\bullet}}$ is a normal subgroup of
$\Gamma_{\rho_{\bullet}}$, hence $E_{\rho_{\bullet}}/K_{\rho_{\bullet}}$ is a Galois extension with Galois group    $\Gamma_{\rho_{\bullet}}/\Delta_{\rho_{\bullet}}$. In particular, $\gamma(E_{\rho_{\bullet}}) = E_{\rho_{\bullet}}$ for all $\gamma \in \Gamma_{\rho_{\bullet}}$.

\begin{prop}
Let $\rho_{\bullet} = (\rho_{\Lambda})_{\Lambda}$ be a compatible system and assume that $E_{\rho_{\bullet}}$ contains the square roots of the values of $\psi$.
Then for each prime $\Lambda$ of $L$ such that $\rho_{\Lambda}$ is residually absolutely irreducible, the equivalence class $[\rho_{\Lambda}]$ contains a representation that can be defined over the field  $(E_{\rho_{\bullet}})_{\Lambda}$ and $(E_{\rho_{\bullet}})_{\Lambda}$ is the smallest such field. Moreover, $E_{\rho_{\bullet}}$ is generated
over $K$ by the set $\{ a_{\p} :\p \;prime \; of \; $F$ \;not \;in \; S\}$.
\end{prop}
\begin{proof}
This is just Proposition 4.3.b of \cite{adw} with $n=2$. 
\end{proof}

Let $\rho_{\bullet} = (\rho_{\Lambda})_{\Lambda} $ be a compatible system.
For each prime $\Lambda$ of $L$ (resp. $\lambda$ of $K$ below $\Lambda$) we denote by $L_{\Lambda}$ (resp. $K_{\lambda}$) the completion of $L$ (resp. $K$) at $\Lambda$ (resp. at $\lambda$).
Consider the Galois extension $L_{\Lambda}/K_{\lambda}$ and define
$\Gamma_{\Lambda} := \Gal(L_{\Lambda}/K_{\lambda})$,
$\mathcal{E}_{\Lambda} := \{ \epsilon: G_F \rightarrow L_{\Lambda}^{\times} \}$ the set of continuous characters from $G_F$ to  $L_{\Lambda}^{\times} $ and
 $\mathcal{G}_{\Lambda}:= \mathcal{E}_{\Lambda} \rtimes \Gamma_{\Lambda}$.
On the other hand, we know that for the equivalence class $\rho_{\Lambda}$ the stabilizer group  $\mathcal{G}_{[\rho_{\Lambda}]}$ of $[\rho_{\Lambda}]$ is of the form $\mathcal{G}_{[\rho_{\Lambda}]}= \mathcal{E}_{[\rho_{\Lambda}]} \rtimes \Gamma_{[\rho_{\Lambda}]}$.

\begin{prop}
Let $\rho_{\bullet}=(\rho_{\Lambda})_{\Lambda}$ be a compatible system and assume that $L$ contains the square roots of the values of $\psi$. Then for each prime $\Lambda$ of $L$ such that $\rho_{\Lambda}$ is residually absolutely irreducible, the projective field of definition of $[\rho_{\Lambda}]$ is the completion of $K_{\rho_{\bullet}}$ at the prime $\lambda$\footnote{Note that in this case $\lambda$ denotes a prime of $K_{\rho_{\bullet}}$ and not of $K$ as above.} below $\Lambda$, i.e. $K_{[\rho_{\Lambda}]} = (K_{\rho_{\bullet}})_{\lambda}$.
\end{prop}
\begin{proof}
This is Theorem 4.5 of \cite{adw}. 
\end{proof}

Given a compatible system $\rho_{\bullet} = (\rho_{\Lambda})_{\Lambda}$ we can also  talk about the residual representations $\overline{\rho}_{\Lambda}$. If $M$ is a local field, we denote by $\kappa(M)$ its residue field. Let $\Lambda$ be a prime of $L$ and assume that $\rho_{\Lambda}$ is defined over $L_{\Lambda}$. We consider the Galois extension $\kappa(L_{\Lambda})/ \kappa(K_{\lambda})$ with Galois group
$\overline{\Gamma}_{\Lambda}$. Moreover, for the equivalence class $[\overline{\rho}_{\Lambda}]$ of the residual representation $\overline{\rho}_{\Lambda}$ we can define $\mathcal{G}_{[\overline{\rho}_{\Lambda}]}$, $\mathcal{E}_{[\overline{\rho}_{\Lambda}]}$ and $\Gamma_{[\overline{\rho}_{\Lambda}]}$.

\begin{prop}\label{chi}
Let $\rho_{\bullet}=(\rho_{\Lambda})_{\Lambda}$ be a compatible system and assume that $L$ contains the square roots of the values of $\psi$. Assume that the restriction to the inertia group $I_{\mathfrak{L}}$ of $[\overline{\rho}_{\Lambda}]$ for the primes $\mathfrak{L}$ of $F$ lying over the residual characteristic of $\Lambda$ is  as in Lemma $\ref{mio}$ or Lemma $\ref{mio2}$. Moreover, assume that there is an integer $k$, independent of $\Lambda$, such that the representations have tame inertia weights at most $k$. Then for all primes $\Lambda$ of $L$, except possibly finitely many, the projective field of definition of  $[\overline{\rho}_{\Lambda}]$ is  $\kappa((K_{\rho_{\bullet}})_{\lambda})$.
\end{prop}
\begin{proof}
The proof is analogous to the proof of theorem 4.6 of \cite{adw} if we replace Proposition 3.3 of \cite{adw} by Lemma $\ref{mio}$ and Lemma $\ref{mio2}$. 
\end{proof}

\begin{rema}
Note that $\psi$ has finite order is a condition needed to ensure that all $\epsilon$ occurring in the inner twists are of finite order and the condition on the square roots of the values of $\psi$ ensure that $\epsilon$ take its values in $E_{\rho_{\bullet}}$. Moreover, the absolute irreducibility condition is needed to ensure that the representations are determined by the characteristic polynomials of Frobenius and the condition on the shape above $\ell$ is needed to exclude that the residual inner twists ramify at $\ell$.
\end{rema}

%%%%%%%%%%%%%%%%%%%%%%%%%%%%%%%%%%%%%%%%%%%%%%%%%%%%%%%%%%%%%%%%%%%%%%%%%%

\section{Galois representations and Hilbert modular forms}
\label{sec:3}

Let $F$ be a totally real field of degree $d$. Denote by $J_F$ the set of all embeddings of $F$ into  $\overline{\Q} \subseteq \C$. An element $k=\sum_{\sigma \in J_F} k_{\sigma} \sigma \in \Z [ J_F ]$ is called a \emph{weight}. We always assume that the $k_{\sigma}$ have the same parity and are all $\geq 2$.
We put $k_0:= \max \{ k_{\sigma} : \sigma \in J_F \} $.
Let $\n$ be an ideal of $\mathcal{O}_{F}$ and $\psi$ be a Hecke character of conductor dividing $\n$ with infinity type $2-k_0$. Consider a Hilbert modular newform $f \in S_k(\n, \psi)$ over $F$.
By a theorem of Shimura \cite{shi}, the Fourier coefficients $a_{\p}(f)$ of $f$, where $\p$ is a prime of $F$, generate a number field $E_f$.

By the work of Ohta, Carayol, Blasius-Rogawski, Wiles and Taylor \cite{ta1} and the local Langlands correspondence for $\GL_2$ (see \cite{car}), we can associate to $f$ a 2-dimensional \emph{strictly compatible system} of Galois representations $\rho_f = (\rho_{f,\iota})$ of $G_F$. Specifically, following Khare and Wintenberger \cite{kw1} $\rho_f$ consists of the data of:
\begin{enumerate}
\item For each rational prime $\ell$ and each embedding $\iota=\iota_{\ell}:E_f \hookrightarrow \overline{\Q}_{\ell}$ a continuous semisimple representation
\[
\rho_{f,\iota} : G_F \rightarrow \GL_2(\overline{\Q}_{\ell}),
\]
\item For each prime $\q$ of $F$, a Frobenius semisimple Weil-Deligne representation $r_{\q}$ with values in $\GL_2(E_f)$ such that:
\begin{enumerate}
\item $r_\q$ is unramified for all $\q$ outside a finite set, 
\item for each rational prime $\ell$, for each prime $\q \nmid \ell$ and for each $\iota: E_f \hookrightarrow \overline{\Q}_{\ell}$, the Frobenius semisimple Weil-Deligne representation associated to $\rho_{f,\iota}|_{D_{\q}}$ is conjugated to $r_{\q}$ via the embedding $\iota$.
\end{enumerate}
\item and a condition for the case where $\q|\ell$ not used in this work.
\end{enumerate}
The compatible system $\rho_f=(\rho_{f,\iota})$ is associated to $f$ in the sense that for each prime $\q \nmid \n \ell$, the characteristic polynomial of $\rho_{f,\iota}(\Frob_{\q})$ is
\[
X^2 - \iota(a_{\q}(f))X+ \iota( \chi (\q)N_{F/\Q}(\q))
\]

Now we introduce a description of the compatible system $\rho_f$ similar to what we saw in the previous section. Let $\iota:E_f \hookrightarrow \overline{\Q}_{\ell}$ be an embedding. Denote by
$E_{\iota}$ the closure of $\iota(E_f)$ and by $\mathcal{O}_{\iota}$ the closure of $\iota(\mathcal{O}_{E_f})$ in $\overline{\Q}_{\ell}$.
Let $(\pi)$ be the maximal ideal of the local ring
$\mathcal{O}_\iota$. Then $\Lambda := \mathcal{O}_{E_f} \cap \iota^{-1}((\pi)) $ is a maximal ideal of $\mathcal{O}_{E_f}$ above $\ell$ and $E_{\iota}$ can be identified with $E_{f,\Lambda}$, the completion of $E_f$ at $\Lambda$ with ring of integers $\mathcal{O}_{f,\Lambda}=(\mathcal{O}_{E_f})_{\Lambda}$. Therefore,
we can identify $\rho_{f,\iota}$ with the $\Lambda$-adic representation
\[
\rho_{f,\Lambda} : G_F \rightarrow \GL_2(\mathcal{O}_{f,\Lambda})
\]
More precisely, the composition of $\rho_{f,\Lambda}$ with the natural inclusion $\mathcal{O}_{f,\Lambda}\hookrightarrow \overline{\Q}_{\ell}$ equals $\rho_{f,\iota}$.
Moreover, $\rho_{f,\Lambda}$ is absolutely irreducible, totally odd and unramified outside $\n \ell$.

Let $\kappa (E_{f,\Lambda}) = \mathcal{O}_{E_f} / \Lambda = \F_{\Lambda}$ be the residual field of $E_{f,\Lambda}$. By taking a Galois stable $\mathcal{O}_{E_f}$-lattice, we define the mod $\ell$
representation
\[
\overline{\rho}_{f,\Lambda}: G_F \rightarrow \GL_2(\F_{\Lambda}),
\]
whose semi-simplification is independent of the particular choice of a lattice. Recall that, according to the first section, $\overline{\rho}_{f,\Lambda}^{\proj}$ denotes the projective quotient of $\overline{\rho}_{f,\Lambda}$, i.e.,
$\overline{\rho}_{f,\Lambda}$ composed with the natural projection $\GL_2(\F_{\Lambda}) \rightarrow \PGL_2(\F_{\Lambda})$.

\begin{theo}\label{did}
Let $\rho_f=(\rho_{f,\Lambda})$ be a compatible system associated  to a Hilbert modular newform  $f\in S_k(\n, \psi)$ without complex multiplication. Let $K_{\rho_f}=K_f$ be the projective field of definition of $\rho_f$, $\lambda = \Lambda \cap K_f$ and $\kappa(K_f)= \mathcal{O}_{K_f}/\lambda = \F_{\lambda}$. Then for almost all $\Lambda$ the image $\overline{\rho}^{\proj}_{f, \Lambda} (G_F)$ is either
 $\PSL_2 (\F_{\lambda})$ or $\PGL_2(\F_{\lambda})$.
\end{theo}
\begin{proof} By a result of Taylor (proposition 1.5, \cite{ty2}) $\overline{\rho}_{f,\Lambda}$ is absolutely irreducible for almost all $\Lambda$.
Then, the result follows directly from the work of Dimitrov (proposition 3.8, \cite{dim}) and Proposition \ref{chi}. 
\end{proof}

\begin{defi}\label{exep}
We say that a prime $\Lambda$ of $E_f$ is  \emph{nonexceptional} if $\overline{\rho}^{\proj}_{f,\Lambda} (G_F)$  is non-solvable and isomophic to $\PSL_2 (\F_{\ell^s})$ or to $\PGL_2(\F_{\ell^s})$ for some $s>0$.
\end{defi}

In particular, if we keep the assumptions of theorem \ref{did} we have only a finite number of exceptional primes. Then, according to the classification of finite subgroups of  $\PGL_2(\overline{\F}_{\ell})$, one has for each exceptional prime that $\overline{\rho}^{\proj}_{f,\Lambda} (G_F)$ is, up to semi-simplification, either an abelian group, a dihedral group, $A_4$, $S_4$ or $A_5$.

%%%%%%%%%%%%%%%%%%%%%%%%%%%%%%%%%%%%%%%%%%%%%%%%%%%%%%%%%%%%%%%%%%%%%%%%%%

\section{Tamely dihedral representations}
\label{sec:4}

In this section we extend the definition of tamely dihedral representation of \cite{diu} to totally real fields in order to exclude complex multiplication and inner twists.

Let $F$, $E$ be  number fields, $\q$ be a prime of $F$ with residual characteristic $q$ and $F_{\q}$ be the completion of $F$ at $\q$. Recall that a 2-dimensional Weil-Deligne representation of $F_{\q}$ with values in $E$ can be described as a pair $(r,N)$, where $r:
W_{\q} \rightarrow \GL_2(E)$ is a continuous representation of the Weil group $W_{\q}$ of $F_{\q}$ for the discrete topology on $\GL_2(E)$ and $N$ is a nilpotent endomorphism of $E^2$ satisfying a 
commutativity relation with $r$ (see \cite{tat}). We denote by $F_{\q^2}$ the unique unramified degree 2 extension of $F_{\q}$ and by $W_{\q^2}$ the Weil group of $F_{\q^2}$.

\begin{defi}\label{mch}
Let $F$ a totally real field, $q$ a rational prime which is completely split in the Hilbert class field of $F$ and $\q$ a prime of $F$ above $q$. A 2-dimensional Weil-Deligne representation $r_{\q} = (r, N)$ of $F_{\q}$ with values in $E$ is called \emph{tamely dihedral of order }$n$ if $N=0$ and there is a tame character 
\[
\varphi:W_{\q^2} \rightarrow E^{\times}
\] 
whose restriction to the inertia group $I_{\q}$ is of niveau 2 (i.e. it factors over $\F_{q^2}^\times$ and not over $\F_q^\times$) and of order $n>2$ such that 
\[
r \cong \Ind ^{W_{\q}}_{W_{\q^2}}(\varphi).
\]
We say that a Hilbert modular newform $f$ is \emph{tamely dihedral of order} $n$ \emph{at the prime} $\q$ if the Weil-Deligne representation $r_{\q}$  associated to the restriction to $D_\q$ of the compatible system $\rho_f = (\rho_{f,\iota})$ is tamely dihedral of order $n$.
\end{defi}

Henceforth, when we talk about the notion of tamely dihedral at a prime $\q$ of $F$, we will assume that $\q$ divides a rational prime $q$ which is completely split in the Hilbert class field of $F$.

If the compatible system $\rho_f = (\rho_{f,\iota})$ is tamely dihedral of order $n$ at $\q$, then for all $\iota : E_f \rightarrow \overline{\Q}_{\ell}$ with $\ell \neq q$, the restriction of $\rho_{f,\iota}$ to $D_{\q}$ is of the form 
\[
\Ind^{F_{\q}}_{F_{\q^2}}(\iota \circ \varphi):= \Ind^{\Gal(\overline{F}_{\q}/F_{\q})} _{\Gal(\overline{F}_{\q}/F_{\q^2})}(\iota \circ \varphi).
\]

Let $\overline{\varphi}_{\Lambda}$ be the reduction of $\varphi$ modulo $\Lambda$ which is a character of the same order as $\varphi$. If $\ell$ and $n$ are relatively prime, then
\[
\overline{\rho}_{f,\Lambda} |_{D_{\q}} = \Ind^{F_{\q}}_{F_{\q^2}}(\overline{\varphi}_{\Lambda}).
\]
Moreover if $n=p^{r}$ for some odd rational prime $p$, then $N_{F/\Q}(\q) \equiv -1 \mod p$, since the character is of niveau $2$.

The next Lemma illustrates how we can avoid the "small" exceptional primes of Galois representations, by assuming that its restriction to $W_\q$ is a tamely dihedral for some appropriate prime ideal $\q$ (i.e. with certain local ramification behavior).

\begin{lemm}\label{ejem}
Let $F$ be a totally real field and $p$, $q$, $\ell$ be distinct odd rational primes.
Let $\q$ be a prime of $F$ above $q$, $\n$ be an ideal of $F$ such that $N_{F/\Q}(\n)$ is relatively prime to $pq$, and $\p_1, \ldots, \p_m $ be the primes with residual characteristic different from $q$ and smaller than or equal to the maximum of $\ell$ and the greatest prime divisor of $N_{F/\Q}(\n)$. 
Let $\rho:G_F \rightarrow \GL_2(\overline{\Q}_\ell)$ be a Galois representation of conductor $\mathfrak{n}$ such that its restriction to $W_\q$ is tamely dihedral of order $p$ at $\q$.
Assume that $\q$ is split in $F(\p_1, \ldots, \p_m)$, and that $p^r$ is unramified in $F$ and greater than the maximum of $5$, $\ell$ and the greatest prime divisor of $N_{F/\Q}(\n)$. 
Then the image of $\overline{\rho}^{\proj}$ is $\PSL_2(\F_{\ell^s})$ or $\PGL_2(\F_{\ell^s})$ for some $s>0$.
\end{lemm}

\begin{proof}
By definition $\overline{\rho} \vert _{I_\q}$ is of the form 
$\left(
  \begin{array}{cc}
    \varphi & 0 \\
    0      & \varphi^q \\
  \end{array}
\right)$, where $\varphi$ is a character of $I_\q$ of order a power of $p \vert q+1$. Then, as $p$ does not divide $q-1$, $\overline{\rho} \vert _{D_\q}$ is irreducible and then so is $\overline{\rho}$. As $p$ is greater than $5$ we have that the projective image can not be $A_4$, $S_4$ or $A_5$.

Now, suppose that the projective image is a dihedral group, i.e. $\overline{\rho}^{\proj} \cong \Ind_{K}^{F}(\alpha)$ for some character $\alpha$ of $\Gal (\overline{F}/K)$, where $K$ is a quadratic extension of $F$. From the ramification of $\overline{\rho}$ we know that $K \subseteq F(\q,\p_1, \ldots, \p_m)$ (because the primes above $\ell$ are contained in $\{\p_1, \ldots, \p_m \}$). As $\ell$ is different from $p$ and $q$ we have that 
\[
\overline{\rho}^{\proj} |_{D_{\q}} \cong \Ind_{F_{\q^2}}^{F_{\q}}(\varphi) \cong \Ind_{K_{\mathfrak{Q}}}^{F_{\q}}(\alpha)
\] 
for some prime $\mathfrak{Q}$ of $K$ above $q$, where $\varphi$ is a niveau 2 character of order $p^r$. From this we have that, if $K$ were ramified at $\q$, then $\overline{\rho}^{\proj}(I_{\q})$ would have even order, but it has order a power of $p$, then the field $K$ is unramified at $\q$. Thus $K \subseteq F(\p_1, \ldots, \p_m)$ and we conclude from the assumptions that $\q$ is split in $K$, which is a contradiction by the irreducibility  of $\overline{\rho}_{f,\Lambda}|_{D_{\q}}$. 
Then according to Dickson's classification the image of $\overline{\rho}^{\proj}$ is $\PSL_2(\F_{\ell^s})$ or $\PGL_2(\F_{\ell^s})$ for some $s>0$.

\end{proof}

More results of this kind will be introduced in section $7$. Now we will show some results similar to those of Section 4 of \cite{diu}, that we will use later in order to exclude nontrivial inner twists.

\begin{lemm}\label{44}
Let $K$ be a topological field and $F$ be a totally real field. Let $\q$ be a prime of $F$, $\epsilon : G_{F_{\q}} \rightarrow K^{\times}$ be a character
 and $\rho: G_{F_{\q}} \rightarrow \GL_2(K)$ be a representation. If the conductors of $\rho$ and of $\rho \otimes \epsilon$ both divide $\q$, then $\epsilon$ or $\epsilon \det (\rho)$ is unramified.
\end{lemm}
\begin{proof}
By the definition of the conductor, $\rho | _{I_{\q}}$ is of the form $\left(
\begin{array}{cc}
        1 & * \\
        0 & \delta \\
\end{array}
\right)$
where $\delta = \det (\rho)|_{I_{\q}}$. Consequently,  $\rho \otimes \epsilon|_{I_{\q}}$ looks like $\left(
\begin{array}{cc}
         \epsilon & * \\
         0 & \epsilon \delta \\
\end{array}
\right).$ Again, by the definition of the conductor, either $\epsilon|_{I_{\q}}$ is trivial or $\epsilon \delta |_{I_{\q}}$ is. 
\end{proof}

\begin{lemm}\label{42}
Let $K$ be a topological field and $F$ a totally real  field. Let $q$ be a rational prime which is completely split in $F$, $\q$ a prime of $F$ above $q$ and $n>2$ an integer relatively prime to $q(q-1)$. Let $\epsilon : G_{F_{\q}} \rightarrow K^{\times}$ and $\varphi, \varphi': \Gal(\overline{F}_{\q}/F_{\q^2}) \rightarrow K^{\times}$ be characters. Assume that $\varphi$ and $\varphi'$ are both of order $n$. If
\[
\Ind^{F_{\q}}_{F_{\q^2}}(\varphi) \cong \Ind^{F_{\q}}_{F_{\q^2}}(\varphi ') \otimes \epsilon ,
\]
then $\epsilon$ is unramified.
\end{lemm}
\begin{proof}
Note that the order of 
$\epsilon|_{\Gal(\overline{F}_{\q}/F_{\q^2})}$ divides $n$. If $\epsilon$ were ramified, the order of $\epsilon|_{I_{\q}}$ would divide $q-1$ times a power of $q$. But this contradicts the fact that $n$ is relatively prime to $q(q-1)$. 
\end{proof}

\begin{lemm}\label{41}
Let $f \in S_k(\n, \psi)$ be a Hilbert modular newform, $\iota: E_f \rightarrow \overline{\Q}_{\ell}$ and  $\rho_f = (\rho_{f,\iota})$ a compatible system associated to $f$. Then, for all inner twists
 $(\gamma, \epsilon) \in \mathcal{G}_{\rho_f}$ one has
\[
\rho_{f,\iota} \otimes \epsilon \cong \rho_{f,\iota \circ \gamma}.
\]
\end{lemm}
\begin{proof}
We know that the traces of any Frobenius element at any unramified prime $\p$ are equal: \[
\Tr((\rho_{f,\iota} \otimes \epsilon)(\Frob_{\p})) = \iota(a_{\p}(f) \epsilon(\Frob_{\p})) = \iota(\gamma (a_{ \p}(f))) = \Tr(\rho_{f, \iota \circ \gamma} (\Frob_{\p})),
\]
from which the result follows. 
\end{proof}

Note that when $\gamma$ is trivial, we are covering the complex multiplication case.

\begin{theo}\label{45}
Let $f \in S_k(\n, \psi)$ be a Hilbert modular newform.
\begin{enumerate}
\item Let $\q$ be a prime of $F$ such that $\q \parallel \n$ and assume that $\psi$ is unramified at $\q$. Then any inner twist of $f$ is unramified at $\q$.
\item Let $\q$ be a prime of $F$ such that $\q^2 \parallel \n$ and $f$ is tamely dihedral at $\q$ of odd order $n > 2$ such that $n$ is relatively prime to
$q(q-1)$. Then any inner twist of $f$ is unramified at $\q$.
\end{enumerate}
\end{theo}
\begin{proof}
 $(i)$ By Lemma $\ref{41}$ the conductors at $\q$ of $\rho_{f,\iota}$ and $\rho_{f,\iota \circ \gamma}$ both divide $\q$. Then, from Lemma $\ref{44}$ we have that $\gamma$ is unramified at $\q$, since the determinant of the representation is unramified at $\q$.

$(ii)$ If $r_{\q}$ is tamely dihedral of order $n$ at $\q$, $\rho_{f,\iota} |_{D_{\q}}$ is of the form
$\Ind^{F_{\q}}_{F_{\q^2}}(\iota \circ \varphi)$, and similarly for $\rho_{f, \iota \circ \gamma}$. Then, by Lemma \ref{41} and Lemma \ref{42} $\gamma$ is unramified at $\q$. 
\end{proof}

\begin{coro}\label{46}
Let $f \in S_k(\n)$ be a Hilbert modular newform over a totally real field $F$ with odd class number such that for every prime $\q | \n$,
\begin{enumerate}
\item $\q \parallel \n$ or
\item $\q^2 \parallel \n$ and $f$ is tamely dihedral at $\q$ of order $n > 2$ such that $(n,q(q-1))=1$.
\end{enumerate}
Then $f$ does not have nontrivial inner twists no complex multiplication.
\end{coro}
\begin{proof}
By Theorem \ref{45} any inner twist is everywhere unramified then these are characters of the Galois group $G:=\Gal(H/F)$ of the Hilbert class field $H$ of $F$. Moreover, as the central character of $f$ is trivial, the field of definition $E_f$ of $f$ is totally real and any inner twist nontrivial is necessarily quadratic.

On the other hand, it is well known that the character group $\hat{G}$ of a finite abelian group $G$ is isomorphic to the original group $G$. Consequently, as $h_F$ is odd, we have that $\hat{G}$ does not have elements of order $2$. Therefore, $f$ does not have nontrivial inner twists. The same happens for complex multiplication. 
\end{proof}

\begin{rema}\label{cin}
When the class number of $F$ is even we may have nontrivial inner twists. This follows from the fact that any finite abelian group is a direct sum of cyclic groups. Then if $h_F$ is even, the Galois group $G$ of the Hilbert class field of $F$ has at least one cyclic group $C$ of even order as direct summand. Thus $C$ has a character of order $2$ which extends to a quadratic character of $G$ by sending $g \in G-C$ to 1.
Moreover, we have an upper bound for the number of nontrivial inner twists which is $2^{\nu_2(h_F)}-1$, where $\nu_2(\cdot)$ is the $2$-adic valuation. Note that we could have $2^{\nu_2(h_F)}-1$ nontrivial inner twists only when
$G \cong (\Z/2 \Z )^{\nu_2(h_F)} \oplus \Z/ m_1 \Z \oplus \cdots \oplus \Z/m_r \Z$ for some divisors $m_i$ of $h_F$.
\end{rema}

%%%%%%%%%%%%%%%%%%%%%%%%%%%%%%%%%%%%%%%%%%%%%%%%%%%%%%%%%%%%%%%%%%%%%%%%%%

\section{Construction of tamely dihedral representations}
\label{sec:5}

In this section we provide a method to construct Hilbert modular newforms which are tamely dihedral at some prime via level raising theorems.

Let $F$ be a totally real field and $f \in S_k(\n, \psi)$ be a Hilbert modular newform over $F$ of level $\n$ and weight $k=\sum_{\sigma \in J_F} k_{\sigma} \sigma$. Let 
\[
\rho_{f,\iota_p} : G_F \rightarrow \GL_2(\overline{\Q}_{p}) 
\]
the $p$-adic Galois representation attached to $f$ as in Section \ref{sec:3}. If $\p$ is a prime of $F$ above $p$ it is well known that $\rho_{f,\iota_p} |_{G_{F_{\p}}}$ is de Rham with $\sigma$-Hodge-Tate weights $\{ m_{\sigma}, m_{\sigma} + k_{\sigma}-1 \}$, where $\sigma \in J_F$ is an embedding lying over $\p$ and $m_{\sigma}=\frac{k_0-k_{\sigma}}{2}$.
Moreover, if $p>k_0$ is unramified in $F$ and relatively prime to $\n$, then $\rho_{f,\iota_p}|_{G_{F_{\p}}}$ is crystalline (see \cite{bro}).

We say that a Galois representation $\rho:G_F \rightarrow \GL_2(\overline{\Q}_p)$ is \emph{modular} (of level $\n$ and weight $k$) if it is isomorphic to $\rho_{f,\iota_p}$ for some Hilbert modular newform $f\in S_{k}(\n, \psi)$ and some embedding $\iota_p: E_f \hookrightarrow \overline{\Q}_p$. We say that a Galois representation $\rho:G_F \rightarrow \GL_2(\overline{\Q}_{p})$ is \emph{geometric} if it is unramified outside of a finite set of primes of $F$ and if for each prime $\p$ above $p$, $\rho|_{G_{F_{\p}}}$ is de Rham.

The main ingredient in the modern proofs of level raising theorems is to have an appropriate modularity lifting theorem as follows.

\begin{theo}[MLT]\label{mlt}
Let $F$ be a totally real field and $p>3$ be a rational prime unramified in $F$. Let $E/\Q_p$  be a finite extension containing the images of all embeddings $F \hookrightarrow \overline{E}$.
Let $\rho,\rho_0: G_F \rightarrow \GL_2(\mathcal{O}_E)$ be two Galois representations such that $\overline{\rho} = \rho\mod \mathcal{P} = \rho_0\mod \mathcal{P}$ for the maximal ideal $\mathcal{P}$ of $\mathcal{O}_E$.
Assume that $\rho_0$ is modular and that $\rho$ is geometric. Assume furthermore that the following properties hold.
\begin{enumerate}
\item $\SL_2(\F_p) \subseteq Im (\overline{\rho})$.
\item For all $\p$ above $p$, $\rho|_{G_{F_{\p}}}$ and $\rho_0|_{G_{F_{\p}}}$ are crystalline.
\item For all $\sigma: F \hookrightarrow E$, the elements of $HT_{\sigma}(\rho)$ differ by at most $p-2$.
\item For all $\sigma:F \hookrightarrow E$, $HT_{\sigma}(\rho)=HT_{\sigma}(\rho_0)$, and contains two distinct elements.
 \end{enumerate}
Then $\rho$ is modular.
\end{theo}
\begin{proof}
The proof is given in Section 5 of \cite{gee}. 
\end{proof}

Now we are ready to state the level raising theorem that we need. This is well known to the experts, but we sketch the proof for lack of a reference.

\begin{theo}\label{lev}
Let $F$ be a totally real field and $E/\Q_p$ be a finite extension sufficiently large. Let $f \in S_k(\n, \psi)$ be a Hilbert modular newform and $p>k_0+1$ be a rational prime unramified in $F$ not dividing  $N_{F/\Q}(\n)$. Moreover, we assume that $\SL_2(\F_p) \subseteq Im(\overline{\rho}_{f,\iota_p})$.
Let $q$ be a rational prime which is completely split in the Hilbert class field of $F$ and $\q$ be a prime of $F$ above $q$ such that $\q \nmid \n$, $N\q \equiv -1 \mod p$ and $\Tr (\overline{\rho}_{f,\iota_p} (\Frob_{\q}))=0$.
Then there exists a Hilbert modular newform $g \in S_k(\n \q^2, \tilde{\psi})$, with $\tilde{\psi}$ having the same conductor as $\psi$, such that
$\overline{\rho}_{f, \iota_p} \cong \overline{\rho}_{g, \iota'_p}$ and $g$ is tamely dihedral of order $p^r$ for some $r>0$ at $\q$. 
\end{theo}
\begin{proof}[Sketch of proof]
Let $S$ be a finite set of places of $F$ consisting of the infinite places, the primes above $p$, the primes dividing $\n$ and the prime $\q$ given above. Let $\overline{\rho}=\rho_{f,\iota_p} \mod \mathcal{P}$, we want to construct a lift $\rho$ of $\overline{\rho}$ such that:
\begin{enumerate}
\item for all places of $S - \{ \q, \p|p \}$, $\rho$ has the same inertial types of $\rho_{f,\iota_p}$,
\item for all $\p$ above $p$, $\rho|_{G_{F_{\p}}}$ is cristalline and has the same Hodge-Tate weights that $\rho_{f,\iota_p}$ and
\item for $\q$, $\rho|_{D_{\q}}$ have supercuspidal inertial type.
\end{enumerate}

Now we will rephrase the problem in terms of universal Galois deformation rings. Indeed, the representation $\rho$ that we want, corresponds to a $\overline{\Q}_p$-point on an appropriate Galois deformation ring $R_{S}^{\mbox{\scriptsize univ}}$ given by choosing the inertial types and Fontaine-Laffaille condition as above. See Section 3 of \cite{ge1} and Section 10 of \cite{kw2} for the precise definition of Galois deformation ring of prescribed type.
Thus it is enough to check that $R_{S}^{\mbox{\scriptsize univ}}$ has a $\overline{\Q}_p$-point.

To prove this, by Proposition 2.2 of \cite{kw2}, it is enough prove that $\dim R_{S}^{\mbox{\scriptsize univ}} \geq 1$ and that $R_{S}^{\mbox{\scriptsize univ}}$ is finite over $\mathcal{O}_E$.
As the image of $\overline{\rho}$ is non-solvable we can conclude that $\delta$ in the formula of remark 5.2.3.a of \cite{boe} is 0, then from Theorem 5.4.1 of \cite{boe} we have that $\dim R_{S}^{\mbox{\scriptsize univ}} \geq 1$. On the other hand, by Section 4.22 of \cite{gee} and Lemma 2.2 of \cite{tayl} to prove that $R^{\mbox{\scriptsize univ}}_{S}$ is finite over $\mathcal{O}_E$, it is enough to show that $R_{S'}^{\mbox{\scriptsize univ}}$ is finite over $\mathcal{O}_E$, where $S'$ is a base change of $S$ as in Section 5.4 of $\cite{gee}$. Then after this base change we can write $R_{S'}^{\mbox{\scriptsize univ}} = R_{\varnothing}^{\mbox{\scriptsize univ}}$.

Therefore, the problem is reduced to showing that $R_{\varnothing}^{\mbox{\scriptsize univ}}$
is finite over $\mathcal{O}_E$. But this is proved in $\cite{gee}$ (see the proof of theorem 5.1). Then we have that the desired lift exists.
Moreover, as this lift satisfies all conditions of Theorem \ref{mlt} we have that $\rho$ is modular. Observe that from conditions on the lift, i)-iii), and compatibility with the local Langlands correspondence, the Hilbert modular newform $g$ corresponding to $\rho$ must be of level $\n\q^2$ and weight $k$ (see also Theorem 1.5 of \cite{jar}). Moreover, condition i) on the lift implies that the central character of $f$ and $g$ agree locally at any prime. 
To see that $g$ is tamely dihedral of order $p^r$ at $\q$ we can translate word by word the proof of Corollary 2.6 of \cite{wie}. 
\end{proof}

The following result shows that there is a set of primes $\q$ of $F$ with positive density to which we can apply Theorem \ref{lev}.

\begin{lemm}\label{53}
Let $\p_1, \ldots, \p_m$ be primes of $F$ and let $p$ be a rational prime unramified in $F$ such that $\p_i \nmid p$ for all $i=1, \ldots m$ and $p \equiv 1\mod 4$.
Let $ \overline{\rho}_p^{\proj} : G_F \rightarrow \PGL_2(\overline{\F}_p)$ be an odd Galois representation with image $\PSL_2(\F_{p^s})$ or $\PGL_2(\F_{p^s})$ such that the image of any complex conjugation is contained in $\PSL_2(\F_{p^s})$. Then, the set of primes $\q$ of $F$ such that
\begin{enumerate}
\item $N\q \equiv -1 \mod p$,
\item $\q$ is split in $F(\p_1, \ldots , \p_m)$ and
\item $\overline{\rho}_p^{\proj} (\Frob_{\q}) \sim \overline{\rho}_p^{\proj}(c)$, where $c$ is any complex conjugation,
\end{enumerate}
has a positive density.
\end{lemm}
\begin{proof}
As in \cite{diu}, the proof is adapted from Lemma 8.2 of Khare and Wintenberger \cite{kw1}. Let $K/F$ be such that $\Gal(\overline{F}/K) = \ker(\overline{\rho}_p^{\proj})$. Then $\Gal(K/F)$ is isomorphic either to $\PGL_2(\F_{p^s})$ or $\PSL_2(\F_{p^s})$. Let $L=K \cap F(\zeta_p)$. Note that $K$ and $F(\zeta_p)$ are linearly disjoint over $L$, and $L/F$ is an extension of degree at most 2 because  $\PSL_2(\F_{p^s})$ is an index 2 simple subgroup of $\PGL_2(\F_{p^s})$.
By assumption the image of any complex conjugation lies in $\Gal(K/L) \cong \PSL_2(\F_{p^s})$. Then, for linear disjointness, we may appeal to Chebotarev's Density Theorem to pick up a set of primes $\q$ of $F$ with positive density such that
$\overline{\rho}_p^{\proj} (\Frob_{\q}) \sim \overline{\rho}_p^{\proj}(c)$, $\q$ is split in $L/F$ and $N_{F/\Q}(\q) \equiv -1\mod p$.
Moreover, we can assume that $\q$ is split in $F(\p_1, \ldots, \p_m)$ without losing the positive density. 
\end{proof}

The next result shows that we can add more than one tamely dihedral prime to Hilbert modular newforms without affecting the local behavior of the other primes​​.

\begin{prop}\label{54}
Let $f \in S_k( \n )$ be a Hilbert modular newform over a totally real field $F$ with odd class number such that no prime divisor of $\n$ divides $2$ and for all $\mathfrak{l} | \n$ either $\mathfrak{l} \Vert \n$ or $\mathfrak{l}^2 \Vert \n$ and $f$ is tamely dihedral at $\mathfrak{l}$ of order $n_{\mathfrak{l}} > 2$. Let $\{ \p_1, \ldots , \p_m \}$ be any finite set of primes of $F$. Then for almost all primes $p \equiv 1 \mod 4$ unramified in $F$ there is a set $S$ of primes of $F$ with positive density which are split in $F(\p_1, \cdots ,\p_m)$ such that for all $\q \in S$ there is a Hilbert modular newform $g \in S_k(\n \q^2)$ which is tamely dihedral at $\q$ of order $p$ and for all $\mathfrak{l}^2 \Vert \n$, $g$ is tamely dihedral at $\mathfrak{l}$ of order $n_{\mathfrak{l}} > 2$.
\end{prop}
\begin{proof}
For $p$ we may choose any prime $p \equiv 1 \mod 4$ unramified in $F$ which is greater than $N_{F/\Q}(\n)$, greater than $k_0+1$, relatively prime to all $n_{\mathfrak{l}}$ and such that $\SL_2(\F_p)\subseteq Im(\overline{\rho}_{f,\iota_p})$ (It can be chosen due to Theorem \ref{did}).

As $-1$ is a square in $\F^{\times}_p$ (because $p \equiv 1 \mod 4$) and there are no nontrivial inner twists (by Corollary \ref{46}) any complex conjugation necessarily lies in $\PSL_2$. Then we can take as $S$ the subset of primes $\q$ of the set provided by Lemma \ref{53} such that $\q$ is over a rational prime $q$ that is completely split in the Hilbert class field of $F$ which has positive density by Chevotarev's Density Theorem.

For any $\q \in S$, Theorem \ref{lev} provides us a Hilbert modular newform $g \in S_k(\n \q^2, \psi)$ tamely dihedral at $\q$ of order $p^r>1$ such that 
\[
\overline{\rho}_{f, \iota_p} \cong \overline{\rho}_{g, \iota'_p}.
\] 
In fact, from this isomorphism, it follows that $r=1$ and that $\psi$ is trivial. The result now follows exactly as in Theorem 5.4.ii of  \cite{diu} by using the isomorphism mentioned above. 
\end{proof}

%%%%%%%%%%%%%%%%%%%%%%%%%%%%%%%%%%%%%%%%%%%%%%%%%%%%%%%%%%%%%%%%%%%%%%%%%%

\section{Hilbert modular newforms without exceptional primes}
\label{sec:6}

Keeping the same notation as in the previous section we will construct families of Hilbert modular newforms without exceptional primes and without nontrivial inner twists.

\begin{prop}\label{61}
Let $p$, $q$, $t$, $u$ be distinct odd rational primes such that $p$ and $t$ are unramified in $F$ and $q$ and $u$ are completely split in the Hilbert class field of $F$. Let $\n$ be an ideal of $F$ relatively prime to $ptqu$ and $\q$, $\mathfrak{u}$ be primes of $F$ above $q$
and $u$ respectively. Let $f \in S_{k} ( \n \q^2 \mathfrak{u}^2)$ be a Hilbert modular newform of weight $k \in \Z[J_F]$ without complex multiplication which is tamely dihedral of order $p^r>5$ at $\q$ and
tamely dihedral of order $t^s>5$ at $\mathfrak{u}$. Let $\p_1, \ldots, \p_m $ be the primes with residual characteristic different from $q$ and $u$ and smaller than or equal to the maximum of $2k_0-1$ and the greatest prime divisor of $N_{F/\Q}(\n \mathfrak{d}_F)$.
Assume that $\q$ is split in $F(\mathfrak{u}, \p_1, \ldots, \p_m)$ and that $\mathfrak{u}$ is split in $F(\q, \p_1 ,\ldots, \p_m)$.
Then $f$ does not have exceptional primes, i.e. for every prime $\Lambda$ of $E_f$ the image of $\overline{\rho}^{\proj}_{f,\Lambda}$
is $\PSL_2(\F_{\ell^s})$ or $\PGL_2(\F_{\ell^s})$ for some $s>0$.
\end{prop}
\begin{proof}
Let $\Lambda$ be any prime of $E_f$ lying over $\ell$. As in Lemma \ref{ejem} the tamely dihedral behavior implies that $\overline{\rho}_{f,\Lambda}$ is irreducible. Because if $\ell \notin \{p,q \}$ then  $\overline{\rho}_{f,\Lambda}|_{D_{\q}}$ is irreducible and if $\ell \in \{p,q \}$ then $\ell \notin \{t,u\}$ and $\overline{\rho}_{f,\Lambda}|_{D_{\mathfrak{u}}}$ is irreducible.

Suppose that the projective image is a dihedral group,
i.e. $\overline{\rho}_{f,\Lambda}^{\proj} \cong \Ind_{K}^{F}(\alpha)$ for some character $\alpha$ of $\Gal (\overline{F}/K)$, where $K$ is a quadratic extension of $F$. From the ramification of $\overline{\rho}_{f,\Lambda}$ we know that $K \subseteq F(\ell, \q, \mathfrak{u},\p_1, \ldots, \p_m)$.
First, we assume that $\ell \notin \{ p, q\}$, then we have that
\[
\overline{\rho}_{f,\Lambda}^{\proj} |_{D_{\q}} \cong \Ind_{F_{\q^2}}^{F_{\q}}(\varphi) \cong \Ind_{K_{\mathfrak{Q}}}^{F_{\q}}(\alpha)
\]
for some prime $\mathfrak{Q}$ of $K$ above $q$, where $\varphi$ is a niveau 2 character of order $p^r>5$. From this we have that, if $K$ were ramified at $\q$, then $\overline{\rho}_{f,\Lambda}^{\proj} (I_{\q})$ would have even order, but it has order a power of $p$, then the field $K$ is unramified at $\q$. For the primes above $\ell$ we have two cases. If $\ell$ is greater than the maximum of $2k_0-1$ and the greatest prime divisor of $N_{F/\Q}(\n \mathfrak{d}_F)$ we have from Lemma 3.4 of \cite{dim} (whose proof works for any totally real field $F$ i.e. even for not necessarily Galois fields) that the field $K$ cannot ramify at the primes of $F$ above $\ell$, then $K \subseteq F(\mathfrak{u},\p_1, \ldots, \p_m)$. Thus, we conclude from the assumptions that $\q$ is split in $K$, which is a contradiction by the irreducibility  of $\overline{\rho}_{f,\Lambda}|_{D_{\q}}$. On the other hand, if $\ell$ is smaller than or equal to the maximum of $2k_0-1$ and the greatest prime divisor of $N_{F/\Q}(\n \mathfrak{d}_F)$ we have that the primes above $\ell$ are contained in the set $\{ \p_1, \ldots, \p_m \}$. Thus $K \subseteq F(\mathfrak{u},\p_1, \ldots, \p_m)$ and we obtain a contradiction as in the previous case.

Now if $\ell \in \{ p,q \}$, in particular $\ell  \notin \{ t, u \}$. Then, exchanging the roles $\q \leftrightarrow \mathfrak{u}$, $p \leftrightarrow t$ and $r \leftrightarrow s$, the same arguments again lead to a contradiction. Then, the image of $\overline{\rho}^{\proj}_{f,\Lambda}$ cannot be a dihedral group.

By the classification of the finite subgroups of $\PGL_2(\overline{\F}_{\ell})$, it remains to exclude $A_4$, $S_4$, $A_5$. But, the image of $\overline{\rho}_{f,\Lambda}^{\proj}$ cannot be any of these groups, since there is an element of order greater than $5$ in the projective image. 
\end{proof}

\begin{rema}\label{62}
In fact from Theorem \ref{did} we can conclude in the previous proposition that for almost all $\Lambda$ the image of $\overline{\rho}^{\proj}_{f,\Lambda}$ is $\PSL_2(\F_{\lambda})$ or $\PGL_2(\F_{\lambda})$.
Moreover, if $f$ has no nontrivial inner twists the image of $\overline{\rho}^{\proj}_{f,\Lambda}$ is $\PSL_2(\F_{\Lambda})$ or $\PGL_2(\F_{\Lambda})$ for almost all primes $\Lambda$ of $E_f$. On the other hand, for the finite set of primes not satisfying this property the image is also large enough because it contains an element of order $p$ or an element of order $t$ (see Lemma 3.1 of \cite{wie}).
\end{rema}

\begin{theo}\label{chido}
There exist families of Hilbert modular newforms $(f_n)_{n\in \N}$ of weight $k$ and trivial central character over a totally real field $F$ with odd class number and without nontrivial inner twists and without complex multiplication such that
\begin{enumerate}
\item for all $n$, all primes $\Lambda_n$ of $E_{f_n}$ are nonexceptional and
\item for a fixed rational prime $\ell$, the size of the image of $\overline{\rho}^{\proj}_{f_n, \Lambda_n}$ is unbounded for running $n$.
\end{enumerate}
\end{theo}

\begin{proof}
Let $f \in S_k(\n)$ of squarefree level $\n$. Since the class number of $F$ is odd $f$ does not have any nontrivial inner twist nor complex multiplication by Corollary \ref{46}.

Let $\{ \p_1, \ldots, \p_m \}$ be the set of primes with norm smaller than or equal to the maximum of  $2k_0-1$ and the greatest prime divisor of $N_{F/\Q}(\n \mathfrak{d}_F)$. Let $B>0$ be any bound and $p$ be any prime greater than $B$ provided  by Proposition \ref{54} applied to $f$ and the set $\{ \p_1, \ldots , \p_m \}$, so that we get $g \in S_k(\n \q^2)$ which is tamely dihedral at $\q$ of order $p$ and which does not have any nontrivial inner twist nor complex multiplication (by Corollary \ref{46}), for some choice of $\q$.

Now applying Proposition \ref{54} to $g$ and the set of primes with norm smaller than or equal to the maximum of  $2k_0-1$ and the greatest prime divisor of $N_{F/\Q}(\n \q^2 \mathfrak{d}_F )$ we obtain a prime $t>B$ different from $p$ and a Hilbert modular newform $h \in S_k(\n \q^2 \mathfrak{u}^2)$ which is tamely
dihedral at $\mathfrak{u}$ of order $t$ and which again does not have any nontrivial inner twist nor complex multiplication (by Corollary \ref{46}), for some choice of $\mathfrak{u}$. Finally, by Proposition \ref{61}, $h$ does not have any exceptional primes.  

We obtain the family $(f_n)_{n \in \N}$ by increasing the bound $B$ step by step, so that elements of greater and greater projective orders appear in the images of inertia groups. 
\end{proof}

We say that a weight $k \in \Z[J_F]$ is \emph{non-induced}, if there do not exist a strict subfield $F'$ of $F$ and a weight $k' \in \Z[J_{F'}]$ such that for each $\sigma \in J_F$, $k_{\sigma} = k'_{(\sigma |_{F'})}$. Note that if $k$ is non-induced then $k$ is not parallel. Moreover, these two conditions are equivalent if the degree $d$ of $F$ is a prime number (see Remark IV.6.3.ii of \cite{dim1}). Moreover, if we assume that $F$ is a Galois field of odd degree this assumption excludes the case where $f$ comes from a base change of a strict subfield of $F$ (see Corollary 3.18 of \cite{dim}). 

\begin{rema}
If we assume, in the previous construction, that $F$ is a Galois field of odd degree and $f$ has a non-induced weight $k$ then the family $(f_n)_{n \in \N}$ of Hilbert modular newforms in Theorem \ref{chido} is such that each $f_n$ does not come from a base change of a strict subfield of $F$ since $g$ and $h$ have the same weight as $f$ and thus, by construction, $g$ and $h$ also do not come from a base change of a strict subfield of $F$.
\end{rema}

%%%%%%%%%%%%%%%%%%%%%%%%%%%%%%%%%%%%%%%%%%%%%%%%%%%%%%%%%%%%%%%%%%%%%%%%%%

\section{A construction via inertial types}
\label{sec:07}

In this section we will explain another method to construct Hilbert modular newforms which are tamely dihedral. This method depends on the main result of \cite{jwe}.

Let $F_\p$ be a finite extension of $\Q_p$, where $F$ is a totally real field and $\p$ is a prime of $F$ above $p$. The local Langlands correspondence establishes a bijection $\pi_\p \mapsto \mbox{rec}(\pi_\p)$ between the set Irr$(\GL_2(F_\p))$ of isomorphism classes of complex-valued irreducible admissible representations of $\GL_2(F_\p)$ and the set WD$(W_\p)$ of isomorphism classes of two-dimensional Frobenius-semisimple Weil-Deligne representations of $F_\p$ preserving $L$-functions and epsilon factors.
In \cite{hen} it is shown that if $\pi_\p \in \mbox{Irr}(\GL_2(F_\p))$, then $\pi_\p \vert _{\GL_2(\mathcal{O}_{F_\p})}$ contains an irreducible finite-dimensional subspace $\tau(\pi_\p)=\tau_\p$ of $\GL_2(\mathcal{O}_{F_\p})$, called the \emph{(local) inertia type} of $\pi_\p$, which characterizes the restriction of $\mbox{rec}(\pi_\p)$ to the inertia group of $\mathcal{O}_{F_\p}$. 

We will denote by $\mathcal{T}(F_\p)$ the set of isomorphism classes of representations of $\GL_2(\mathcal{O}_{F_\p})$ which arise as inertial types for members of Irr$(\GL_2(F_\p))$. 
We say that $\tau_\p$ is a supercuspidal (resp. special, principal series) type if $\pi_\p$ is supercuspidal (resp. special, principal series). We define the quantity
\[
d(\tau_\p) = (-1)^{\alpha} q^{\beta}(\gamma q +1)(\delta q - 1)
\]
where $q$ is the cardinality of the residue field of $F_\p$ and the values of $\alpha, \beta, \gamma, \delta \in \Z_{\geq 0}$ are determined by the type of $\pi_\p$ (see Section 2.1 of \cite{jwe} for details).

On the other hand, let $k \geq 2$ and $w$ be two integers of the same parity and let $\mathcal{D}_{k,w}$ be the essentially discrete series representation of $\GL_2(\R)$ as in paragraph 0.2 of \cite{car}. Then the central character of $\mathcal{D}_{k,w}$ is $t \mapsto t^{-w}$. We will denote by $\mathcal{T}(\R)$ the set of all such representations $\mathcal{D}_{k,w}$ and we simply define $\tau(\mathcal{D}_{k,w}) = \mathcal{D}_{k,w}$. We also define $d(\mathcal{D}_{k,w})=k-1$.

Let $S_{f}$ be the set of primes of $F$ and recall that $J_F$ denote the set of all embeddings of $F$ into $\overline{\Q} \subseteq \C$. Given a cuspidal automorphic representation $ \pi = \pi_\infty \otimes \pi_f$ (where $\pi_\infty = \otimes_{\sigma \in J_F} \pi_\sigma $ and $\pi_f = \otimes_{\p \in S_{f}} \pi_\p$) of $\GL_2(\A_F)$ arising from a Hilbert modular form over $F$ of weight $k = \sum_{\sigma \in J_F} k_\sigma \sigma \in \Z [J_F]$, we can associate to it, the representation 
\[
\tau(\pi) = \bigotimes \limits_{\sigma \in J_F} \tau(\pi_\sigma) \bigotimes \limits_{\p \in S_{\tiny fin}} \tau(\pi_\p) 
\] 
of $\GL_2( \hat{\mathcal{O}}_F \times (F \otimes \R))$. In particular if $\pi_\sigma \cong \mathcal{D}_{k_\sigma,w_\sigma}$, then $w_\sigma = w_\sigma'$ for all $\sigma, \sigma' \in J_F$ and the integers $k_\sigma$ and $w_\sigma$  all have the same parity.  
Moreover, we have that $\tau(\pi_\p)$ is the trivial representation for all primes of $F$ not dividing the level of $\pi$ and the central character of $\pi$ is an algebraic Hecke character of $\A^*_F$ whose restriction to $\mathcal{O}_{F_\p}$ (resp. $F_\tau^* \cong \R$) is equal to the central character of $\tau(\pi_\p)$ for all $\p \in S_{f}$ (resp. of $\tau(\pi_\sigma)$ for all $\sigma \in J_F$).

Accordingly, we define the set $\mathcal{T}(F)$ of \emph{global inertial types} to consist of the collections $\tau = (\tau_v)_{v\in J_F \cup S_{f}}$ satisfying the conditions:
\begin{enumerate}
\item For all but finitely many $v$, $\tau_v$ is the trivial representation.
\item There exists an algebraic Hecke character of $\A_F$ whose component at each $v$ agrees with the central character of $\tau_v$. 
\end{enumerate}

For each $\tau \in \mathcal{T}(F)$ we define
\[
d(\tau) = \prod_{v\in J_F \cup S_{f}} d(\tau_v),
\]
the product makes sense because all but finitely many factors are 1.

Clearly, if $\pi$ is a cuspidal automorphic representation of $\GL_2(\A_F)$ arising from a Hilbert modular form over $F$, $\tau(\pi)$ belongs to $\mathcal{T}(F)$.  Then a natural question is: given an arbitrary global inertial type $\tau \in \mathcal{T}(F)$, when does this type come from a Hilbert modular form? The answer, provided by Weinstein, is as follows.

Let $H(\tau)$ be the set of cuspidal automorphic representations $\pi$ of $\GL_2(\A_F)$ arising from a Hilbert modular form over $F$ for which $\tau(\pi)=\tau$. The main result of \cite{jwe} establishes that
\[
\# H(\tau) = 2^{1-(\#J_F)} | \zeta_F(-1) | h_F d(\tau) + O(2^{\nu(\tau)}),
\]
where $\zeta_F(s)$ is the Dedekind zeta function for $F$ and $\nu(\tau)$ is the number of primes $\p$ where $\tau_\p$ is nontrivial. Finally, by comparing the quantity $d(\tau)$ with the error term $2^{\nu(\tau)}$, we can obtain the following result.

\begin{theo}[Weinstein]\label{wei}
Up to twisting by 1-dimensional characters, the set of global inertial types $\tau \in \mathcal{T}(F)$ for which $H(\tau)= \emptyset$ is finite.
\end{theo}
\begin{proof}
This is just Corolary 1.2 of \cite{jwe} 
\end{proof}
Let $p$, $q$, $t$, $u$ be distinct odd rational primes, such that $p$ and $t$ are unramified in $F$ and $q$ and $u$ are completely split in the Hilbert class field of $F$. Let $\n$ be an ideal  of $F$ squarefree and relatively prime to $pqtu$ and  $\q$ and $\mathfrak{u}$ be a primes of $F$ above $q$ and $u$ respectively. By Theorem $\ref{wei}$ we can ensure the existence of a Hilbert modular newform $f \in S_k(\n \q^2 \mathfrak{u}^2)$ with supercuspidal types in $\q$ and $\mathfrak{u}$ for some choice of a prime $\q$ (or equivalently of a prime $\mathfrak{u}$) large enough (because in this case we have that $d(\tau_{\q}) = q(q-1)$ and $d(\tau_{\mathfrak{u}}) = u(u-1)$). We note that the hypothesis of making a choice of a prime $\q$ or $\mathfrak{u}$ large enough can be avoided if we have large enough weights or enough primes ramified.  Thus, for an appropriate choice, $f$ is tamely dihedral of order $p^r>5$ at $\q$ and tamely dihedral of order $t^s>5$ at $\mathfrak{u}$ (see the proof of Corollary 2.6 of \cite{wie}). Then we have the following general result.

\begin{theo}
For any totally real field $F$ and any weight $k \in \Z[J_F]$ there exist families of Hilbert modular newforms $(f_n)_{n\in \N}$ over $F$ of weight $k$, trivial central character and without complex multiplication such that
\begin{enumerate}
\item for all $n$, all primes $\Lambda_n$ of $E_{f_n}$ are nonexceptional and
\item for a rational fixed prime $\ell$, the size of the image of $\overline{\rho}^{\proj}_{f_n, \Lambda_n}$ is unbounded for running $n$.
\end{enumerate}
Moreover, if $F$ is a Galois field of odd degree then the elements of $(f_n)_{n\in \N}$ do not come from a base change of a strict subfield of $F$ for all $n$.
\end{theo}
\begin{proof}
Let $B>0$ be some bound. Let $p$ and $t$ be rational primes (as above) greater than $B$. By choosing the prime $\q$ (or the prime $\mathfrak{u}$) of $F$ large enough (or alternatively by choosing an ideal $\n$ of $F$ with sufficient prime divisors) we have, from the previous discussion about Weinstein's result and Proposition \ref{61}, that for every weight $k \in \Z[J_F]$ there exists a Hilbert modular newform of weight $k$ without exceptional primes. 
Thus, by increasing the bound $B$ we obtain a family $(f_n)_{n\in \N}$ such that elements of greater and greater orders appear in the inertia images because the number of inner twists is bounded and depends only on the class number of $F$ (see remark \ref{cin}).

Finally, if $F$ is a Galois fields and $k \in \Z[J_F]$ is a non-induced weight we have, from Corollary 3.18 of $\cite{dim}$, that $f_n$ does not come from a base change of a strict subfield of $F$ for all $n$. 
\end{proof}

%%%%%%%%%%%%%%%%%%%%%%%%%%%%%%%%%%%%%%%%%%%%%%%%%%%%%%%%%%%%%%%%%%%%%%%%%%


\begin{thebibliography}{80}

\addcontentsline{toc}{section}{Referencias}

\bibitem[ADW17]{adw} S. Arias-de-Reyna, L. Dieulefait and G. Wiese. \textit{Compatible systems of symplectic Galois representations and the inverse Galois problem I}. Trans. Amer. Math. Soc. 369(2), 887-908 (2017).

\bibitem[Bo13]{boe} G. Bockle. \textit{Deformation of Galois representations}. In: Darmon, H., Diamond, F., Dieulefait, L.V., Edixhoven, B., Rotger, V. (eds.) Elliptic Curves, Hilbert Modular Forms and Galois Deformations, pp 21-115 . Progress in Math., Birkhauser (2013).

\bibitem[Bre99]{bro} C. Breuil. \textit{Une remarque sur les repr$\grave{\mbox{e}}$sentations locales $p$-adiques et les congruences entre formes modulaires de Hilbert}. Bull. Soc. Math. France 127, 459-472 (1999).

\bibitem[Car86]{car} H. Carayol. \textit{Sur les repr$\grave{\mbox{e}}$sentations $\ell$-adiques associ$\acute{\mbox{e}}$es aux formes modulaires de Hilbert}. Ann. Sci. $\acute{\mbox{E}}$cole Norm. Sup. S$\acute{\mbox{e}}$r. 4, 19 no. 3, 409-468 (1986).

\bibitem[Dim03]{dim1} M. Dimitrov. \textit{Valeur critique de la fonction L adjointe d'une forme modulaire de Hilbert et arithm$\acute{\mbox{e}}$tique du motif correspondant}. PhD thesis, Universit$\acute{\mbox{e}}$ Paris 13 (2003).

\bibitem[Dim05]{dim} M. Dimitrov. \textit{Galois representations modulo $p$ and cohomology of Hilbert modular varieties}. Ann. Sci. $\acute{\mbox{E}}$cole Norm. Sup. S$\acute{\mbox{e}}$r. 4, 38 no. 4, 505-551 (2005).

\bibitem[Di11]{diu} L. Dieulefait and G. Wiese. \textit{On Modular Forms and the Inverse Galois Problem}. Trans. Amer. Math. Soc. 363 no. 9, 4569-4584 (2011).

\bibitem[Ge11]{ge1} T. Gee. \textit{Automorphic lifts of prescribed types}. Math. Ann. 350 no. 1, 107-144 (2011).

\bibitem[Ge13]{gee} T. Gee. \textit{Modularity lifting theorems}. In: AWS. http://swc.math.arizona.edu/aws/2013/ Accessed Nov 2013.

\bibitem[He02]{hen} G. Henniart. \textit{Sur l'unicit$\acute{\mbox{e}}$ des types pour $\GL(2)$}. Duke Math. J. 155 no. 2, 298-310 (2002).

\bibitem[Jar99]{jar} F. Jarvis. \textit{Level lowering for modular mod $\ell$ Galois representations over totally real fields}. Math. Ann. 313, no. 1, 141-160 (1999).

\bibitem[KW09a]{kw1} C. Khare and J-P. Wintenberger. \textit{Serre's modularity conjecture I}. Invent. Math. 178 no. 3, 485-504 (2009).

\bibitem[KW09b]{kw2} C. Khare and J-P. Wintenberger. \textit{Serre's modularity conjecture II}. Invent. Math. 178 no. 3, 505-586 (2009).

\bibitem[Shi78]{shi} G. Shimura. \textit{The special values of the zeta functions associated with Hilbert modular forms}. Duke Math. J. 45 no. 3, 637-679 (1978).

\bibitem[Ta79]{tat} J. Tate. \textit{Number Theoretic Background}. Proc. in Sym. in Pure Math. Vol. 33 part 2, 3-26 (1979).

\bibitem[Tay89]{ta1} R. Taylor. \textit{On Galois representations associated to Hilbert modular forms}. Invent. Math. 98  no. 2, 265-280 (1989).

\bibitem[Tay97]{ty2} R. Taylor. \textit{On Galois representations associated to Hilbert modular forms II}. In: Coates, J., Yau, S.T. (eds.) Elliptic Curves, Modular Forms and Fermat's Last Theorem, 2nd Edition pp 333-340. Int. Press (1997).

\bibitem[Tay03]{tayl} R. Taylor. \textit{On icosahedral Artin representations}. Amer. J. Math. 125 no.3, 549-566 (2003).

\bibitem[We09]{jwe} J. Weinstein. \textit{Hilbert modular forms with prescribed ramification}. Int. Math. Res. Not. IMRN no. 8, 1388-1420 (2009).

\bibitem[Wi08]{wie} G. Wiese. \textit{On projective linear groups over finite fields as Galois groups over the rational numbers}. In: van der Geer, G., Endixhoven, B., Moonen, B. (eds.) Modular Forms on Schiermonnikoog pp 343-350. Cambridge Univ. Press, Cambridge (2008)

\end{thebibliography}
\end{document}